\documentclass[11pt]{amsart}
\usepackage{amsmath,amsthm,amsfonts,amssymb,mathrsfs, mathtools}
\usepackage{enumerate}
\usepackage[colorlinks=true, linkcolor=blue, citecolor=magenta, menucolor=black]{hyperref}
\usepackage[left=3.5cm, right=3.5cm]{geometry}
\usepackage[utf8]{inputenc}
\usepackage[demo]{graphicx}
\usepackage[up]{caption}
\usepackage{subcaption}
\usepackage{pgf,tikz}
\usepackage[toc,page]{appendix}
\usetikzlibrary{arrows}
\usetikzlibrary{patterns}
\usepackage{color}

\makeatletter
\@namedef{subjclassname@2020}{%
  \textup{2020} Mathematics Subject Classification}
\makeatother

\usepackage{tikz-cd}

\numberwithin{equation}{section}

\newtheorem{letterthm}{Theorem}

\newtheorem{thm}{Theorem}[section]
\newtheorem{lem}[thm]{Lemma}

\newtheorem{cor}[thm]{Corollary}

\theoremstyle{definition}
\newtheorem{rem}[thm]{Remark}

\newtheorem*{example}{Example}

\newtheorem*{definition}{Definition}

\newcommand{\R}{\mathbb{R}}
\newcommand{\C}{\mathbb{C}}

\newcommand{\N}{\mathbf{N}}

\newcommand{\Fix}{\operatorname{Fix}}

\newcommand{\Aut}{\operatorname{Aut}}

\newcommand{\Prob}{\operatorname{Prob}}
\newcommand{\supp}{\operatorname{supp}}

\newcommand{\rk}{\operatorname{rk}}

\newcommand{\rE}{\operatorname{ E}}
\newcommand{\rC}{\operatorname{C}}
\newcommand{\rL}{\operatorname{L}}

\newcommand{\Har}{\operatorname{Har}}

\allowdisplaybreaks

\begin{document}

\title[The noncommutative factor theorem for lattices in product groups]{The noncommutative factor theorem \\ for lattices in product groups}

\begin{abstract}
We prove a noncommutative Bader-Shalom factor theorem for lattices with dense projections in product groups. As an application of this result and our previous works, we obtain a noncommutative Margulis factor theorem for all irreducible lattices $\Gamma < G$ in higher rank semisimple algebraic groups. Namely, we give a complete description of all intermediate von Neumann subalgebras $\rL(\Gamma) \subset M \subset \rL(\Gamma \curvearrowright G/P)$ sitting between the group von Neumann algebra and the group measure space von Neumann algebra associated with the action on the Furstenberg-Poisson boundary.
\end{abstract}

\author{R\'emi Boutonnet}
\address{Institut de Math\'ematiques de Bordeaux \\ CNRS \\ Universit\'e Bordeaux I \\ 33405 Talence \\ FRANCE}
\email{remi.boutonnet@math.u-bordeaux.fr}
\thanks{RB is supported by ANR grant AODynG 19-CE40-0008}

\author{Cyril Houdayer}
\address{Universit\'e Paris-Saclay \\  Laboratoire de Math\'ematiques d'Orsay\\ CNRS \\ Institut Universitaire de France \\ 91405 Orsay\\ FRANCE}
\email{cyril.houdayer@universite-paris-saclay.fr}
\thanks{CH is supported by Institut Universitaire de France}


\maketitle

\section{Introduction and statement of the main results}

Let $H$ be a locally compact second countable (lcsc) group. A Borel probability measure $\mu \in \Prob(H)$ is said to be {\bf admissible} if $\mu$ is absolutely continuous with respect to the Haar measure, $\supp(\mu)$ generates $H$ as a semigroup, and $\supp(\mu)$ contains a neighborhood of the identity element $e \in H$. A standard probability space $(X, \nu)$ is said to be a $(H, \mu)$-{\bf space} if it is endowed with a nonsingular action $H \curvearrowright (X, \nu)$ for which the probability measure $\nu$ is $\mu$-{\bf stationary}, that is, $\nu = \mu \ast \nu$.

Following \cite{Fu62a, Fu62b}, for every admissible Borel probability measure $\mu \in \Prob(H)$, we denote by $(B, \nu_B)$ the $(H, \mu)$-{\bf Furstenberg-Poisson boundary}. Recall that $(B, \nu_B)$ is the unique $(H, \mu)$-space for which the $H$-equivariant Poisson transform 
$$\rL^\infty(B, \nu_B) \to \Har^\infty(H, \mu) : f \mapsto \left( h \mapsto \int_B f(h b) \, {\rm d}\nu_B(b) \right)$$
is isometric and surjective. Here,  $\Har^\infty(H, \mu)$ denotes the space of bounded (right) $\mu$-harmonic functions on $H$. A $(H, \mu)$-{\bf boundary} $(C, \nu_C)$ is a $H$-equivariant measurable factor of $(B, \nu_B)$. In that sense, $(B, \nu_B)$ is the maximal $(H, \mu)$-boundary. For every $(H, \mu)$-boundary $(C, \nu_C)$, we regard $\rL^\infty(C) \subset \rL^\infty(B)$ as a $H$-invariant von Neumann subalgebra. Note that the center $\mathscr Z(H)$ acts trivially on $(B, \nu_B)$. We refer to \cite{Fu00, BS04} for further details on the Furstenberg-Poisson boundary.

\begin{definition}
We say that the pair $(H, \mu)$ satisfies the {\bf boundary freeness condition} if for every nontrivial $(H, \mu)$-boundary $(C, \nu_C)$ and every $h \in H \setminus \mathscr Z(H)$, we have $\nu_C(\Fix_C(h)) = 0$ where $\Fix_C(h) = \left\{c \in C \mid hc = c \right\}$.
\end{definition}

\begin{example}
Let $k$ be a local field, $\mathbf H$ a $k$-isotropic almost $k$-simple linear algebraic $k$-group and $\mu \in \Prob(\mathbf H(k))$ an admissible Borel probability measure. A combination of \cite[Corollary 5.2]{BS04} and \cite[Lemma 6.2]{BBHP20} shows that the pair $(\mathbf H(k), \mu)$ satisfies the boundary freeness condition.
\end{example}

Recall that the {\bf quasi-center} $\mathscr{Q}\mathscr{Z}(H)$ is the (not necessarily closed) subgroup of all elements $h \in H$ for which the centralizer $\mathscr Z_H(h)$ is open in $H$. We have $\mathscr Z(H) < \mathscr{Q}\mathscr{Z}(H)$.

In order to state our main result, we introduce the following notation. Let $d \geq 2$. For every $i \in \{1, \dots, d\}$, let $G_i$ be a lcsc group and $\mu_i \in \Prob(G_i)$ an admissible Borel probability measure. Denote by $(B_i, \nu_{B_i})$ the $(G_i, \mu_i)$-Furstenberg-Poisson boundary. Set $(G, \mu) = \prod_{i = 1}^d (G_i, \mu_i)$ and $(B, \nu_B) = \prod_{i = 1}^d (B_i, \nu_{B_i})$. Then $(B, \nu_B)$ is the $(G, \mu)$-Furstenberg-Poisson boundary (see \cite[Corollary 3.2]{BS04}). For every $i \in \{1, \dots, d\}$, denote by $p_i : G \to G_i$ and by $\widehat p_i : G \to \prod_{j \neq i} G_j$ the canonical homomorphisms.

\begin{definition}
Let $\Gamma < G$ be a lattice, that is, $\Gamma < G$ is a discrete subgroup with finite covolume. We say that $\Gamma < G$ is {\bf embedded with dense projections} if for every $i \in \{1, \dots, d\}$, the restriction $p_i|_\Gamma : \Gamma \to G_i$ is injective and $\widehat p_i(\Gamma) < \prod_{j \neq i} G_j$ is dense.
\end{definition}

Let $\Gamma < G$ be a lattice embedded with dense projections and assume that its center $\mathscr Z(\Gamma)$ is finite. Note that $\mathscr Z(\Gamma) < \prod_{i = 1}^d \mathscr Z(G_i)$ and so $\mathscr Z(\Gamma)$ acts trivially on $(B, \nu_B)$. Set $\Lambda = \Gamma/\mathscr Z(\Gamma)$. Consider the well-defined ergodic action $\Lambda \curvearrowright (B, \nu_B)$ and denote by $\rL(\Lambda \curvearrowright B)$ its associated {\bf group measure space} von Neumann algebra. Whenever $(C, \nu_C)$ is a $(G, \mu)$-boundary, we regard $\rL(\Lambda \curvearrowright C) \subset \rL(\Lambda \curvearrowright B)$ as a von Neumann subalgebra.

Our main result is the following noncommutative analogue of Bader-Shalom's factor theorem for lattices in product groups (see \cite[Theorem 1.7]{BS04}). We call such a result a {\bf Noncommutative Factor Theorem (NFT)}.

\begin{letterthm}[NFT for lattices in products]\label{thmA}
Keep the same notation as above. For every $i \in \{1, \dots, d\}$, assume that $\mathscr{Q}\mathscr{Z}(G_i) = \mathscr Z(G_i)$ and that the pair $(G_i, \mu_i)$ satisfies the boundary freeness condition. Let $\rL(\Lambda) \subset M \subset \rL(\Lambda \curvearrowright B)$ be an intermediate von Neumann subalgebra. 

Then for every $i \in \{1, \dots, d\}$, there exists a unique $(G_i, \mu_i)$-boundary $(C_i, \nu_{C_i})$ such that with $(C, \nu_C) =  \prod_{i = 1}^d (C_i, \nu_{C_i})$, we have $M = \rL(\Lambda \curvearrowright C)$.
\end{letterthm}

We point out that unlike the proof of the recent NFT for lattices in higher rank simple algebraic $k$-groups obtained in \cite{BBH21} (see also \cite{Ho21}), we cannot rely on the noncommutative Nevo-Zimmer theorem from \cite{BH19, BBH21}. The proof of Theorem \ref{thmA} consists of two steps. Firstly, we show that when $\rL(\Lambda) \subset M \subset \rL(\Lambda \curvearrowright B)$ and $\rL(\Lambda) \neq M$, there exist $i \in \{1, \dots, d\}$ and a nontrivial $(G_i, \mu_i)$-boundary $(C_i, \nu_{C_i})$ such that $\rL(\Lambda \curvearrowright C_i) \subset M$. To do this, we consider the conjugation action $\Lambda \curvearrowright M$, we exploit the assumption that $\mathscr{Q}\mathscr{Z}(G_i) = \mathscr Z(G_i)$ and we use the dichotomy theorem for boundary structures  \cite[Theorem 5.8]{BBHP20} (in lieu of the noncommutative Nevo-Zimmer theorem \cite{BH19, BBH21}). Secondly, we exploit the assumption that the pair $(G_i, \mu_i)$ satisfies the boundary freeness condition and we combine Bader-Shalom's factor theorem \cite{BS04} and Suzuki's result \cite{Su18} to show that there exists a unique $(G, \mu)$-boundary $(C, \nu_C)$ such that $M = \rL(\Lambda \curvearrowright C)$.

Next, we apply Theorem \ref{thmA} to the setting of higher rank lattices. We introduce the following notation. Let $d \geq 1$. For every $i \in \{1, \dots, d\}$, let $k_i$ be a local field, $\mathbf G_i$ a  simply connected $k_i$-isotropic almost $k_i$-simple linear algebraic $k_i$-group and set $G_i = \mathbf G_i(k_i)$. 

\begin{definition}
Set $G = \prod_{i = 1}^d G_i$. We say that $\Gamma < G$ is a {\bf higher rank lattice} if the following conditions are satisfied:
\begin{itemize}
\item [$(\rm i)$] $\Gamma < G$ is a discrete subgroup with finite covolume; 
\item [$(\rm ii)$] If $d \geq 2$, then $\Gamma < G$ is embedded with dense projections;
\item [$(\rm iii)$] $\sum_{i = 1}^d \rk_{k_i}(\mathbf G_i) \geq 2$.
\end{itemize}
\end{definition}

Observe that the higher rank assumption $\sum_{i = 1}^d \rk_{k_i}(\mathbf G_i) \geq 2$ implies that exactly one of the following two situations happens. 
\begin{itemize}
\item Either $d  = 1$ ({\bf simple} case). Then $k_1 = k$ and $\mathbf G_1 = \mathbf G$ is an almost $k$-simple algebraic $k$-group such that $\rk_k(\mathbf G) \geq 2$.
\item Or $d \geq 2$ ({\bf semisimple} or {\bf product} case).
\end{itemize}

Let $\Gamma < G$ be a higher rank lattice. For every $i \in \{1, \dots, d\}$, choose a minimal parabolic $k_i$-subgroup $\mathbf P_i < \mathbf G_i$ and set $P_i = \mathbf P_i(k_i)$. Set $P = \prod_{i = 1}^d P_i$ and endow the homogeneous space $G/P$ with its unique $G$-invariant measure class. Note that $\mathscr Z(\Gamma) < \mathscr Z(G) < P$. Set $\Lambda = \Gamma/\mathscr Z(\Gamma)$.  Consider the well-defined ergodic action $\Lambda \curvearrowright G/P$ and its associated group measure space von Neumann algebra $\rL(\Lambda \curvearrowright G/P)$.

As an application of Theorem \ref{thmA} and \cite[Theorem D]{BBH21}, we obtain the following noncommutative analogue of Margulis'\! factor theorem for all higher rank lattices (see \cite[Theorem IV.2.11]{Ma91}).

\begin{letterthm}[NFT for higher rank lattices]\label{thmB}
Keep the same notation as above. Let $\rL(\Lambda) \subset M \subset \rL(\Lambda \curvearrowright G/P)$ be an intermediate von Neumann subalgebra.

Then for every $i \in \{1, \dots, d\}$, there exists a unique parabolic $k_i$-subgroup $\mathbf P_i < \mathbf Q_i < \mathbf G_i$ such that with $Q = \prod_{i = 1}^d \mathbf Q_i(k_i)$, we have $M = \rL(\Lambda \curvearrowright G/Q)$. In particular, the rank $\sum_{i = 1}^d \rk_{k_i}(\mathbf G_i)$ is an invariant of the inclusion $\rL(\Lambda) \subset \rL(\Lambda \curvearrowright G/P)$.
\end{letterthm}

We actually prove in Theorem \ref{thm-higher-rank} a more general version of Theorem \ref{thmB} where the algebraic $k_i$-group $\mathbf G_i$ may not be simply connected and the lattice $\Gamma < G$ may not be embedded with dense projections. We refer to \cite[Section 6]{BBH21} and \cite[Section 5]{Ho21} for a discussion of the relevance of Theorem \ref{thmB} regarding Connes'\! rigidity conjecture for the group von Neumann algebras of higher rank lattices.

Finally, we apply Theorem \ref{thmA} to the setting of lattices in products of trees. We introduce the following notation. Let $d\geq 2$. For every $i \in \{1, \dots, d\}$, let $\mathsf T_i$ be a bi-regular tree and denote by $\Aut^+(\mathsf T_i)$ the group of bi-coloring preserving automorphisms of $\mathsf T_i$. Let $\Gamma < \prod_{i = 1}^d \Aut^+(\mathsf T_i)$ be a uniform lattice. Denote by $G_i$ the closure of the image of $\Gamma$ in $\Aut^+(\mathsf T_i)$ and assume that $G_i$ is $2$-transitive on the boundary $\partial \mathsf T_i$. Assume that $\Gamma < \prod_{i = 1}^d G_i$ is embedded with dense projections. Endow $\partial \mathsf T_i$ with its unique $G_i$-invariant measure class and $B = \prod_{i = 1}^d \partial \mathsf T_i$ with the product measure class. 

As an application of Theorem \ref{thmA}, we obtain the following noncommutative analogue of Burger-Mozes'\! factor theorem for lattices in product of trees (see \cite[Theorem 4.6]{BM00b}).

\begin{letterthm}[NFT for lattices in products of trees]\label{thmC}
Keep the same notation as above. Let $\rL(\Gamma) \subset M \subset \rL(\Gamma \curvearrowright B)$ be an intermediate von Neumann subalgebra.

Then there exists a unique subset $J \subset \{1, \dots, d\}$ such that with $B_J = \prod_{j \in J} \partial \mathsf T_j$, we have $M = \rL(\Gamma \curvearrowright B_J)$.
\end{letterthm}

{\hypersetup{linkcolor=black} \tableofcontents}

\section{Continuous elements in noncommutative boundaries}

Let $H$ be a lcsc group and $M$ a von Neumann algebra. We say that $M$ is a $H$-{\bf von Neumann algebra} if it is endowed with a continuous action $H \curvearrowright M$. Let $M$ and $N$ be $H$-von Neumann algebras. We say that a unital normal mapping $\Psi : M \to N$ is a $H$-{\bf map} if $\Psi$ is $H$-equivariant with respect to the actions $H \curvearrowright M$ and $H \curvearrowright N$. 

For $i \in \{1, 2\}$, let $G_i$ be a lcsc group and $\mu_i \in \Prob(G_i)$ an admissible Borel probability measure. Denote by $(B_i, \nu_{B_i})$ the $(G_i, \mu_i)$-Furstenberg-Poisson boundary. Set $(G, \mu) = (G_1 \times G_2, \mu_1 \otimes \mu_2)$ and $(B, \nu_B) = (B_1 \times  B_2, \nu_{B_1} \otimes \nu_{B_2})$. Denote by $p_i : G \to G_i$ the canonical homomorphism. Let $\Gamma < G$ be a lattice embedded with dense projections.

Let $M$ be a $\Gamma$-von Neumann algebra endowed with a faithful normal ucp $\Gamma$-map $\Phi : M \to \rL^\infty(B)$. Denote by $\sigma : \Gamma \curvearrowright M$ the action by automorphisms. Following \cite[Definition 5.3]{BBHP20}, an element $x \in M$ is said to be $G_1$-{\bf continuous} if for every sequence $(\gamma_n)_n$ in $\Gamma$ such that $p_1(\gamma_n) \to e$ in $G_1$, we have $\sigma_{\gamma_n}(x) \to x$ $\ast$-strongly in $M$. The subset $M_1 \subset M$ of all $G_1$-continuous elements in $M$ forms a $\Gamma$-invariant von Neumann subalgebra for which the action $\Gamma \curvearrowright M_1$ extends to a continuous action $G \curvearrowright M$ such that $G_2$ acts trivially (see \cite[Theorem 5.5]{BBHP20}).

Our main result relies on the explicit computation of the subalgebra of continuous elements in the specific context of noncommutative boundaries. More precisely, we regard $\rL(\Gamma \curvearrowright B)$ as a $\Gamma$-von Neumann algebra via the conjugation action $\Gamma \curvearrowright \rL(\Gamma \curvearrowright B)$. The canonical conditional expectation $\rE : \rL(\Gamma \curvearrowright B) \to \rL^\infty(B)$ is a faithful normal ucp $\Gamma$-map. We denote by $u_\gamma \in \rL(\Gamma) \subset \rL(\Gamma \curvearrowright B)$ for $\gamma \in \Gamma$ the unitaries implementing the action $\Gamma \curvearrowright (B, \nu_B)$.

\begin{thm}\label{thm-continuous}
For every $i \in \{1, 2\}$, the von Neumann subalgebra of $G_i$-continuous elements in $\rL(\Gamma \curvearrowright B)$ is equal to $\rL(\Gamma_i \curvearrowright B_i)$, where $\Gamma_i = p_i^{-1}(\mathscr{Q}\mathscr{Z}(G_i)) \cap  \Gamma < \Gamma$.
\end{thm}

Note that this theorem extends \cite[Lemma 5.4]{BBHP20} to the whole group measure space von Neumann algebra $\rL(\Gamma \curvearrowright B)$. The general approach is similar to \cite[Lemma 5.4]{BBHP20} and relies on the following refinement of Peterson's result \cite[Lemma 5.1]{Pe14}. We essentially follow his proof, with a little more care to ensure our extra conditions.

\begin{lem}\label{lem-sequence}
As in the statement of Theorem \ref{thm-continuous}, set $\Gamma_1 = p_1^{-1}(\mathscr{Q}\mathscr{Z}(G_1)) \cap  \Gamma$.
Let $\gamma \in \Gamma \setminus \Gamma_1$ and $E \subset B_2$ be a nonnull measurable subset. Then there exists a sequence $(\gamma_n)_{n}$ in $\Gamma$ such that $\nu_{B_2}(p_2(\gamma_n) E) \to 1$, $p_1(\gamma_n) \to e$ and $(\gamma_n \gamma \gamma_n^{-1})_{n}$ are pairwise distinct in $\Gamma$.
\end{lem}

\begin{proof}
Set $g = p_1(\gamma) \in G_1\setminus \mathscr{Q}\mathscr{Z}(G_1)$. Then the closed subgroup $\mathscr Z_{G_1}(g)$ has empty interior in $G_1$. Since $G_1$ is a lcsc group, we may choose a compatible proper right invariant metric $d : G_1 \times G_1 \to \R_+$ (see e.g.\! \cite{St73}).

Set $\gamma_0 = e$. We construct by induction a sequence $(\gamma_n)_{n \in \N^*}$ in $\Gamma$ such that for every $n \in \N^*$, the element $\gamma_n \in \Gamma$ satisfies the following three conditions:
\begin{itemize}
\item [$(\rm i)$] $1 - \nu_{B_2}(p_2(\gamma_n) E) \leq \frac{1}{n}$, 
\item [$(\rm ii)$] $d(p_1(\gamma_n), e) \leq \frac{1}{n}$ and 
\item [$(\rm iii)$] $p_1(\gamma_{n}) \notin \bigcup_{j = 0}^{n - 1} p_1(\gamma_j)\mathscr Z_{G_1}(g)$.  
\end{itemize}
Let $n \geq 0$ and assume that we have constructed $\gamma_0, \dots, \gamma_n \in \Gamma$ satisfying the above three conditions. Let us construct $\gamma_{n + 1} \in \Gamma$ that satisfies the above three conditions. Since the action $G_2 \curvearrowright (B_2, \nu_{B_2})$ is nonsingular, we may choose a compact neighborhood $\mathscr K_2 \subset G_2$ of $e \in G_2$ such that 
$$\forall k \in \mathscr K_2, \quad \|{k}_\ast \nu_{B_2} - \nu_{B_2}\|_{1} \leq \frac{1}{2(n + 1)}.$$ 
Set $\mathscr F = \bigcup_{j = 0}^n p_1(\gamma_j)\mathscr Z_{G_1}(g)$ and note that $\mathscr F$ is a closed set with empty interior in $G_1$. Consider the open ball $B_{G_1}(e, \frac{1}{2(n + 1)})$  of center $e$ and radius $\frac{1}{2(n + 1)}$ in $G_1$. Since $B_{G_1}(e, \frac{1}{2(n + 1)}) \cap \mathscr F^c$ is a nonempty open set, we may choose a closed ball $\mathscr C = \overline B_{G_1}(o, r) \subset B_{G_1}(e, \frac{1}{2(n + 1)}) \cap \mathscr F^c$ of center $o \in G_1$ and radius $0 < r \leq \frac{1}{2(n + 1)}$ in $G_1$. Since $\mathscr C$ is compact and since $\mathscr F$ is closed, we have $\alpha = d(\mathscr C , \mathscr F) > 0$. Note that $\alpha \leq d(\mathscr C, e) \leq \frac{1}{2(n + 1)}$. Next, define $\mathscr K_1 = \overline B_{G_1}(e, \frac{\alpha}{2})$ to be the closed ball of center $e \in G_1$ and radius $\frac{\alpha}{2}$ in $G_1$. For every $y \in \mathscr F$, every $c \in \mathscr C$ and every $k \in \mathscr K_1$, we have
$$d(kc, y) \geq d(c, y) - d(kc, c) = d(c, y) - d(k, e) \geq \alpha - \frac{\alpha}{2} = \frac{\alpha}{2} > 0.$$
Thus, we have  $\mathscr K_1 \mathscr C \subset \mathscr F^c$. Moreover, for every $k \in \mathscr K_1$ and every $c \in \mathscr C$, we have
$$d(kc, e) \leq d(kc, c) + d(c, e) = d(k, e) + d(c, e) \leq \frac{\alpha}{2} + \frac{1}{2(n + 1)} \leq \frac{1}{n + 1}.$$

Denote by $m \in \Prob(G/\Gamma)$ the unique $G$-invariant Borel probability measure. Following the proof of \cite[Lemma 5.1]{Pe14}, since the pmp action $G_2 \curvearrowright (G/\Gamma, m)$ is ergodic, Kakutani's random ergodic theorem implies that for $\mu_2^{\otimes \N}$-almost every $(\omega_j)_j \in G_2^{\N}$ and $m$-almost every $z\Gamma \in G/\Gamma$, we have
$$\lim_N \frac{1}{N + 1} \sum_{j = 0}^N \mathbf 1_{(\mathscr K_1 \times \mathscr K_2)\Gamma}(\omega_j^{-1} \cdots \omega_0^{-1} z) = m((\mathscr K_1 \times \mathscr K_2)\Gamma) > 0.$$
Since $m((\mathscr C \times \mathscr K_2)\Gamma) > 0$, we may choose $z \in \mathscr C \times \mathscr K_2$ such that for $\mu_2^{\otimes \N}$-almost every $(\omega_j)_j \in G_2^{\N}$, the intersection $\{ \omega_j^{-1} \cdots \omega_0^{-1} \mid j \in \N \} \cap (\mathscr K_1 \times \mathscr K_2)\Gamma z^{-1}$ is infinite. Since $(B_2, \nu_{B_2})$ is the $(G_2, \mu_2)$-Furstenberg-Poisson boundary, for $\mu_2^{\otimes \N}$-almost every $\omega = (\omega_j)_j \in G_2^\N$, the limit measure $(\nu_{B_2})_\omega$ is a Dirac mass (see \cite[Theorem 2.14]{BS04}). Moreover, we have $\nu_{B_2} = \int_{G_2^\N} (\nu_{B_2})_\omega\, {\rm d}\mu_2^{\otimes \N}(\omega)$ (see \cite[Theorem 2.10]{BS04}). Therefore, we have 
$$\mu_2^{\otimes \N} (\{(\omega_j)_j \in G_2^\N \mid \lim_j \nu_{B_2}(\omega_j^{-1} \cdots \omega_0^{-1} p_2(z) E) = 1\}) = \nu_{B_2}(p_2(z) E) > 0.$$
We may choose $(\omega_j)_j \in G_2^{\N}$ such that the intersection $\{ \omega_j^{-1} \cdots \omega_0^{-1} \mid j \in \N \} \cap (\mathscr K_1 \times \mathscr K_2)\Gamma z^{-1}$ is infinite and $\lim_j \nu_{B_2}(\omega_j^{-1} \cdots \omega_0^{-1} p_2(z) E) = 1$. We may then choose $j \in \N$, $h \in \mathscr K_1 \times \mathscr K_2$ and $\gamma_{n + 1} \in \Gamma$ such that $\omega_j^{-1} \cdots \omega_0^{-1} = h \gamma_{n + 1} z^{-1}$ and $1 - \nu_{B_2}(\omega_j^{-1} \cdots \omega_0^{-1} p_2(z) E) \leq \frac{1}{2(n + 1)}$. Then $p_1(\gamma_{n + 1}) = p_1(h^{-1})p_1(z) \in \mathscr K_1 \mathscr C$ and so $p_1(\gamma_{n + 1}) \notin \bigcup_{j = 0}^n p_1(\gamma_j)\mathscr Z_{G_1}(g)$ and $d(p_1(\gamma_{n + 1}), e) \leq \frac{1}{n + 1}$. Moreover, we have 
\begin{align*}
1 - \nu_{B_2}(p_2(\gamma_{n + 1}) E) &= 1 - \nu_{B_2}(p_2(h^{-1}) \omega_j^{-1} \cdots \omega_0^{-1} p_2(z) E)  \\
&= 1 - \nu_{B_2}(\omega_j^{-1} \cdots \omega_0^{-1} p_2(z) E) + \|p_2(h)_\ast \nu_{B_2} - \nu_{B_2}\|_1 \\
&\leq \frac{1}{2(n + 1)} + \frac{1}{2(n + 1)} = \frac{1}{n + 1}.
\end{align*}
Thus, the element $\gamma_{n + 1} \in \Gamma$ satisfies the above three conditions. 

The sequence $(\gamma_n)_n$ that we have constructed by induction satisfies the conclusion of the lemma.
\end{proof}

Let $(X, \nu)$ be a standard probability space. For every measurable function $f : X \to \C$, we set 
$$\|f\|_{\nu} = \left( \int_X |f(x)|^2 \, {\rm d}\nu(x)\right)^{1/2}.$$

\begin{proof}[Proof of Theorem \ref{thm-continuous}]
We may assume that $i = 1$. Set $\mathscr B = \rL(\Gamma \curvearrowright B)$ and denote by $\mathscr B_1 \subset \mathscr B$ the von Neumann subalgebra of all $G_1$-continuous elements in $\mathscr B$. Set $\varphi = \nu_B \circ \rE \in \mathscr B_\ast$. 

Observe that $\rL(\Gamma_1 \curvearrowright B_1) \subset \mathscr B_1$. Indeed, it is obvious that $\rL^\infty(B_1) \subset \mathscr B_1$. Let now $\gamma \in \Gamma_1$ so that $p_1(\gamma) \in \mathscr{Q}\mathscr{Z}(G_1)$. Let $(\gamma_n)_n$ be a sequence in $\Gamma$ such that $p_1(\gamma_n) \to e$. Since $\mathscr Z_{G_1}(p_1(\gamma)) < G_1$ is open, there exists $n_0 \in \N$ such that $p_1(\gamma_n) \in \mathscr Z_{G_1}(p_1(\gamma))$ for every $n \geq n_0$. Since $p_1|_\Gamma : \Gamma \to G_1$ is injective, we have $\gamma_n \in \mathscr Z_\Gamma(\gamma)$ for every $n \geq n_0$. This implies that $u_\gamma \in \mathscr B_1$. Altogether, we obtain $\rL(\Gamma_1 \curvearrowright B_1) \subset \mathscr B_1$.

Next, we prove that $\mathscr B_1 \subset \rL(\Gamma_1 \curvearrowright B_1)$. Firstly, we show that $\mathscr B_1 \subset \rL(\Gamma_1 \curvearrowright B)$. By contraposition, let $x \in \mathscr B \setminus\rL(\Gamma_1 \curvearrowright B)$. Write $x = \sum_{\gamma \in \Gamma} x_\gamma u_\gamma$ for the Fourier expansion of $x \in \mathscr B$ where $x_\gamma = \rE(xu_\gamma^*)$ for every $\gamma \in \Gamma$. Since $x \not\in \rL(\Gamma_1 \curvearrowright B)$, there exists $\gamma \in \Gamma \setminus \Gamma_1$ such that $x_\gamma \neq 0$. We may regard $x_\gamma \in \rL^\infty(B) = \rL^\infty(B_2, \rL^\infty(B_1))$ as a measurable function $f : B_2 \to \rL^\infty(B_1)$. Since $x_\gamma \neq 0$, the measurable function $f : B_2 \to \rL^\infty(B_1)$ possesses a nonzero essential value $y \in \rL^\infty(B_1)$. Set $\varepsilon = \|y\|_{\nu_{B_1}}/2 > 0$. Then the measurable subset $E = \{b \in B_2 \mid \|f(b) - y\|_{\nu_{B_1}} < \varepsilon\}$ is nonnull. Choose a sequence $(\gamma_n)_n$ in $\Gamma$ that satisfies the conclusion of Lemma \ref{lem-sequence} for $\gamma \in \Gamma \setminus \Gamma_1$ and $E \subset B_2$. Then 
\begin{align*}
\forall n \in \N, \quad \|u_{\gamma_n} x u_{\gamma_n}^* - x\|_\varphi^2 &= \sum_{h \in \Gamma} \|\sigma_{\gamma_n}(x_h) - x_{\gamma_n h \gamma_n^{-1}}\|_{\nu_B}^2 \\
& \geq \|\sigma_{\gamma_n}(x_\gamma) - x_{\gamma_n \gamma \gamma_n^{-1}}\|_{\nu_B}^2.
\end{align*}
On the one hand, since the elements $(\gamma_n \gamma \gamma_n^{-1})_n$ are pairwise distinct in $\Gamma$ and since $\|x\|_\varphi^2 = \sum_{h \in \Gamma} \|x_h\|_{\nu_B}^2 \geq \sum_{n \in \N} \|x_{\gamma_n \gamma \gamma_n^{-1}}\|_{\nu_B}^2$, we have 
$$\lim_n \|x_{\gamma_n \gamma \gamma_n^{-1}}\|_{\nu_B}^2 = 0.$$ 
On the other hand, since $p_1(\gamma_n) \to e$, since $\nu_{B_2}(p_2(\gamma_n) E) \to 1$ and since $\|f(b)\|_{\nu_{B_1}} \geq \|y\|_{\nu_{B_1}} - \|f(b) - y\|_{\nu_{B_1}} \geq \varepsilon$ for every $b \in E$, we have
\begin{align*}
\liminf_n\|\sigma_{\gamma_n}(x_\gamma)\|_{ \nu_B}^2 &= \liminf_n \int_{B_2} \|\sigma_{p_1(\gamma_n)}(f(p_2(\gamma_n)^{-1} b))\|^2_{\nu_{B_1}}  {\rm d}\nu_{B_2}(b) \\
&= \liminf_n \int_{B_2} \|f(p_2(\gamma_n)^{-1} b)\|^2_{\nu_{B_1} \circ p_1(\gamma_n)} \, {\rm d}\nu_{B_2}(b) \\
&\geq \liminf_n \int_{p_2(\gamma_n) E} \|f(p_2(\gamma_n)^{-1} b)\|^2_{\nu_{B_1} \circ p_1(\gamma_n)} \, {\rm d}\nu_{B_2}(b) \geq \varepsilon^2.
\end{align*}
Altogether, this implies that $\liminf_n \|u_{\gamma_n} x u_{\gamma_n}^* - x\|_\varphi^2 \geq \varepsilon^2$ and so $x \not\in \mathscr B_1$. This shows that $\mathscr B_1 \subset \rL(\Gamma_1 \curvearrowright B)$. 

Secondly, we show that $\mathscr B_1 \subset \rL(\Gamma_1 \curvearrowright B_1)$. Indeed, let $x \in \mathscr B_1$. The previous paragraph shows that $x \in \rL(\Gamma_1 \curvearrowright B)$ and so we may write $x = \sum_{\gamma \in \Gamma_1} x_\gamma u_\gamma$ where $x_\gamma = \rE(x u_\gamma^*)$ for every $\gamma \in \Gamma_1$. Since the faithful normal conditional expectation $\rE : \mathscr B \to \rL^\infty(B)$ is $\Gamma$-equivariant, we have that $\rE(\mathscr B_1)$ is contained in the von Neumann subalgebra of $G_1$-continuous elements in $\rL^\infty(B)$, which is equal to $\rL^\infty(B_1)$ by \cite[Lemma 5.4]{BBHP20}. Since $\rL(\Gamma_1) \subset \mathscr B_1$, we have $x_\gamma = \rE(x u_\gamma^*) \in \rL^\infty(B_1)$ for every $\gamma \in \Gamma_1$. Then \cite[Corollary 3.4]{Su18} further implies that $\mathscr B_1 \subset \rL(\Gamma_1 \curvearrowright B_1)$. Thus, we have $\mathscr B_1 = \rL(\Gamma_1 \curvearrowright B_1)$.
\end{proof}

Next, we further assume that $\mathscr{Q}\mathscr{Z}(G_i) = \mathscr Z(G_i)$ for every $i \in \{1, 2\}$ and that $\mathscr Z(\Gamma)$ is finite. Set $\Lambda = \Gamma/\mathscr Z(\Gamma)$ and denote by $\pi : \Gamma \to \Lambda$ the quotient homomorphism. Consider the well-defined ergodic action $\Lambda \curvearrowright (B, \nu_B)$. We may regard $\rL(\Lambda \curvearrowright B)$ as a $\Gamma$-von Neumann algebra via the quotient homomorphism $\pi : \Gamma \to \Lambda$ and the conjugation action $\Lambda \curvearrowright \rL(\Lambda \curvearrowright B)$. Moreover, the canonical conditional expectation $\rE : \rL(\Lambda \curvearrowright B) \to \rL^\infty(B)$ is a faithful normal ucp $\Gamma$-map. We derive the following result that will be used in the proof of Theorem \ref{thmA}.

\begin{cor}\label{cor-ergodic}
For every $i \in \{1, 2\}$, the von Neumann subalgebra of $G_i$-continuous elements in $\rL(\Lambda \curvearrowright B)$ is equal to $\rL^\infty(B_i)$. Moreover, the action $\Gamma \curvearrowright \rL(\Lambda \curvearrowright B)$ is ergodic.
\end{cor}

\begin{proof}
The second assertion follows from the first one since a $\Gamma$-invariant element in $\rL(\Lambda \curvearrowright B)$ is both $G_1$-continuous and $G_2$-continuous, hence must be contained in $\rL^\infty(B_1) \cap \rL^\infty(B_2) = \C 1$.

To prove the first assertion, we observe that with the above notation, $\rL(\Lambda \curvearrowright B) \cong z \rL(\Gamma \curvearrowright B)$, where $z = \frac{1}{|\mathscr Z(\Gamma)|} \sum_{h \in \mathscr Z(\Gamma)} u_h \in \mathscr Z(\rL(\Gamma \curvearrowright B))$. Let us be more explicit about this identification. Since $\mathscr Z(\Gamma)$ acts trivially on $B$, the projection $z$ is indeed central in $\rL(\Gamma \curvearrowright B)$. Given $g \in \Gamma$, the element $zu_g$ only depends on $\pi(g) \in \Lambda$. Then the map 
$$\rL(\Lambda \curvearrowright B) \to z\rL(\Gamma \curvearrowright B) : a u_{\pi(g)}  \mapsto zau_g, \quad   a \in \rL^\infty(B), g \in \Gamma,$$
is easily seen to extend to the desired von Neumann algebra isomorphism $\Theta : \rL(\Lambda \curvearrowright B) \to z\rL(\Gamma \curvearrowright B)$. Note that $\Theta$ is $\Gamma$-equivariant. 

Let $i \in \{1,2\}$. Since $\mathscr{Q}\mathscr{Z}(G_i) = \mathscr Z(G_i)$ and since $p_i|_\Gamma : \Gamma \to G_i$ is injective, we have $\Gamma_i = p_i^{-1}( \mathscr Z(G_i)) = \mathscr Z(\Gamma)$. It is obvious that all the elements of $\rL^\infty(B_i)$ are $G_i$-continuous in $\rL(\Lambda \curvearrowright B)$. Conversely, let $x \in \rL(\Lambda \curvearrowright B)$ be a $G_i$-continuous element. Then $\Theta(x) \in z\rL(\Gamma \curvearrowright B)$ is $G_i$-continuous in $\rL(\Gamma \curvearrowright B)$. Applying Theorem \ref{thm-continuous}, we obtain that $\Theta(x) \in \rL(\mathscr Z(\Gamma) \curvearrowright B_i) \cap z\rL(\Gamma \curvearrowright B) = z \rL^\infty(B_i)$. Thus, we have $x \in \rL^\infty(B_i)$. This finishes the proof of the corollary.
\end{proof}

\section{Proofs of the main results}

\begin{proof}[Proof of Theorem \ref{thmA}]
Regard $\mathscr B = \rL(\Lambda \curvearrowright B)$ as a $\Gamma$-von Neumann algebra and denote by $\rE : \mathscr B \to \rL^\infty(B)$ the canonical $\Gamma$-equivariant faithful normal conditional expectation. By Corollary \ref{cor-ergodic}, the action $\Gamma \curvearrowright \mathscr B$ is ergodic. Let $\rL(\Lambda) \subset M \subset \mathscr B$ be an intermediate von Neumann subalgebra. Then $M \subset\mathscr B$ is a $\Gamma$-invariant von Neumann subalgebra and the action $\Gamma \curvearrowright M$ is ergodic. Consider the faithful normal ucp $\Gamma$-map $\Phi = \rE|_M : M \to \rL^\infty(B)$. If $\Phi : M \to \rL^\infty(B)$ is invariant, then for every $x \in M$ and every $\lambda \in \Lambda$, we have $\rE(x u_\lambda^*) \in \C 1$. Then we infer that $M = \rL(\Lambda)$ (see e.g.\! \cite[Corollary 3.4]{Su18}).

Next, assume that $\Phi : M \to \rL^\infty(B)$ is not invariant. Then \cite[Theorem 5.8]{BBHP20} implies that there exists $i \in \{1, \dots, d\}$ such that the $\Gamma$-invariant von Neumann subalgebra $M_i \subset M$ of all $G_i$-continuous elements in $M$ is nontrivial. Moreover, the action $\Gamma \curvearrowright M_i$ extends to a continuous action $G_i \curvearrowright M_i$ and the ucp map $\Phi|_{M_i} : M_i \to \rL^\infty(B_i)$ is $G_i$-equivariant and not invariant. Since $M \subset \mathscr B$ is $\Gamma$-invariant, we have $M_i \subset \mathscr B_i$ where $\mathscr B_i \subset \mathscr B$ is the $\Gamma$-invariant von Neumann subalgebra of all $G_i$-continuous elements in $\mathscr B$. Corollary \ref{cor-ergodic} implies that $M_i \subset \rL^\infty(B_i)$. Then $M_i \subset \rL^\infty(B_i)$ is a nontrivial $G_i$-invariant von Neumann subalgebra and so there exists a nontrivial $(G_i, \mu_i)$-boundary $(D_i, \nu_{D_i})$ such that $M_i = \rL^\infty(D_i)$. Then we have $\rL(\Lambda \curvearrowright D_i) \subset M \subset \mathscr B$. Since the pair $(G_i, \mu_i)$ satisfies the boundary freeness condition and since the restriction $p_i |_\Gamma : \Gamma \to G_i$ is injective, it follows that the nonsingular action $\Lambda \curvearrowright (D_i, \nu_{D_i})$ is essentially free. Then a combination of Bader-Shalom's factor theorem \cite[Theorem 1.7]{BS04} and Suzuki's result \cite[Theorem 3.6]{Su18} implies that for every $j \in \{1, \dots, d\}$, there exists a unique $(G_j, \mu_j)$-boundary $(C_j, \nu_j)$ such that with $(C, \nu_C)  = \prod_{j = 1}^d (C_j, \nu_{C_j})$, we have $M = \rL(\Lambda \curvearrowright C)$.
\end{proof}

\begin{rem}
We point out that the analogous statement of Theorem \ref{thmA} for $\rL(\Gamma \curvearrowright B)$ is false in the case when $\mathscr Z(\Gamma)$ is nontrivial. Indeed, the presence of the central projection $z \in \mathscr Z(\rL(\Gamma \curvearrowright B))$ as defined in the proof of Corollary \ref{cor-ergodic} produces pathological intermediate subalgebras such as $z \rL(\Gamma) \oplus (1-z)\rL(\Gamma \curvearrowright B)$. Nevertheless, our strategy still allows to classify intermediate subalgebras in this case.
\end{rem}

Next, we prove a general noncommutative factor theorem for higher rank lattices that will imply Theorem \ref{thmB}. 

We introduce the following notation. Let $d \geq 1$. For every $i \in \{1, \dots, d\}$, let $k_i$ be a local field and $\mathbf G_i$ a $k_i$-isotropic almost $k_i$-simple linear algebraic $k_i$-group and set $G_i = \mathbf G_i(k_i)$. Denote by $G_i^+ < G_i$ the subgroup generated by the subgroups $\mathbf U(k_i)$ where $\mathbf U$ runs through the set of unipotent $k_i$-split subgroups of $\mathbf G_i$ (see \cite[Proposition I.5.4 $(\rm i)$]{Ma91}). Choose a minimal parabolic $k_i$-subgroup $\mathbf P_i < \mathbf G_i$ and set $P_i = \mathbf P_i(k_i)$. 
Set $G = \prod_{i = 1}^d G_i$, $G^+ = \prod_{i = 1}^d G_i^+$, $P = \prod_{i = 1}^d P_i$ and endow the homogeneous space $G/P$ with its unique $G$-invariant measure class. Observe that we have $G = G^+ \cdot P$ and so the action $G^+ \curvearrowright G/P$ is transitive (see \cite[Proposition I.5.4$(\rm vi)$]{Ma91}). For every $i \in \{1, \dots, d\}$, denote by $p_i : G \to G_i$ and by $\widehat p_i : G \to \prod_{j \neq i} G_j$ the canonical homomorphisms.  

Let $\Gamma < G$ be a lattice such that for every $i \in \{1, \dots, d\}$, we have $\prod_{j \neq i} G_j^+ < \overline{\widehat p_i(\Gamma)}$ and $p_i|_\Gamma : \Gamma \to G_i$ is injective. We point out that $\Gamma < G$ need not be embedded with dense projections. Note that $\mathscr Z(\Gamma) < \mathscr Z(G) < P$ (see \cite[Lemma II.6.3 $(\rm I)$]{Ma91}). Set $\Lambda = \Gamma/\mathscr Z(\Gamma)$ and denote by $\pi : \Gamma \to \Lambda$ the quotient homomorphism. Consider the well-defined ergodic action $\Lambda \curvearrowright G/P$ and its associated group measure space von Neumann algebra $\rL(\Lambda \curvearrowright G/P)$. We may regard $\rL(\Lambda \curvearrowright G/P)$ as a $\Gamma$-von Neumann algebra via the quotient homomorphism $\pi : \Gamma \to \Lambda$ and the conjugation action $\Lambda \curvearrowright \rL(\Lambda \curvearrowright G/P)$. 

\begin{lem}\label{lem-conjugation}
The action $\Gamma \curvearrowright \rL(\Lambda \curvearrowright G/P)$ is ergodic.
\end{lem} 

\begin{proof}
Choose a Borel probability measure $\mu \in \Prob(G)$ of the form $\mu = \psi \cdot m_G$ where $m_G$ is a Haar measure on $G$ and $\psi : G \to \R_+$ is a compactly supported continuous function such that the set $\{g \in G \mid \psi(g) > 0 \text{ and } \psi(g^{-1}) > 0\}$ generates $G$. Since $G/P$ is compact, there exists a Borel probability measure $\nu \in \Prob(G/P)$ such that $(G/P, \nu)$ is a $(G, \mu)$-space. Moreover, $\nu\in \Prob(G/P)$ belongs to the unique $G$-invariant measure class (see e.g.\! \cite[Lemma 1.1]{NZ97}). 

By \cite[Proposition VI.4.1]{Ma91}, there exists a fully supported probability measure $\mu_\Gamma \in \Prob(\Gamma)$ such that $(G/P, \nu)$ is a $(\Gamma, \mu_\Gamma)$-space. Denote by $\mu_\Lambda \in \Prob(\Lambda)$ the pushforward measure of $\mu_\Gamma$ under the quotient homomorphism $\pi : \Gamma \to \Lambda$. Then $(G/P, \nu)$ is a $(\Lambda, \mu_\Lambda)$-space. Since $\Lambda$ has infinite conjugacy classes, \cite[Lemma 2.6]{KP21} implies that the conjugation action $\Lambda \curvearrowright \rL(\Lambda \curvearrowright G/P)$ is ergodic. Thus, the action $\Gamma \curvearrowright \rL(\Lambda \curvearrowright G/P)$ is ergodic.
\end{proof}

The following theorem is a noncommutative analogue of Margulis'\! factor theorem for all higher rank lattices (see \cite[Theorem IV.2.11]{Ma91}). It is a more general version of Theorem \ref{thmB}.

\begin{thm}\label{thm-higher-rank}
Keep the same notation as above. Let $\rL(\Lambda) \subset M \subset \rL(\Lambda \curvearrowright G/P)$ be an intermediate von Neumann subalgebra.

Then for every $i \in \{1, \dots, d\}$, there exists a unique parabolic $k_i$-subgroup $\mathbf P_i < \mathbf Q_i < \mathbf G_i$ such that with $Q = \prod_{i = 1}^d \mathbf Q_i(k_i)$, we have $M = \rL(\Lambda \curvearrowright G/Q)$. In particular, the rank $\sum_{i = 1}^d \rk_{k_i}(\mathbf G_i)$ is an invariant of the inclusion $\rL(\Lambda) \subset \rL(\Lambda \curvearrowright G/P)$.
\end{thm}

\begin{proof}
Regard $\mathscr B = \rL(\Lambda \curvearrowright G/P)$ as a $\Gamma$-von Neumann algebra and denote by $\rE : \mathscr B \to \rL^\infty(G/P)$ the canonical $\Gamma$-equivariant faithful normal conditional expectation. By Lemma \ref{lem-conjugation}, the action $\Gamma \curvearrowright \mathscr B$ is ergodic. Let $\rL(\Lambda) \subset M \subset \mathscr B$ be an intermediate von Neumann subalgebra and set $\Phi = \rE|_M : M \to \rL^\infty(G/P)$. Using the same notation and following the same argument as in the first paragraph of the proof of Theorem \ref{thmA}, we conclude that if $\Phi : M \to \rL^\infty(G/P)$ is invariant, then $M = \rL(\Lambda)$. Next, assume that $\Phi : M \to \rL^\infty(G/P)$ is not invariant. There are two cases to consider.

Firstly, assume that $d = 1$. Let $k = k_1$ and $\mathbf G = \mathbf G_1$ such that $\rk_k(\mathbf G) \geq 2$. We proceed as in the proof of \cite[Theorem D]{BBH21}. By \cite[Theorem 5.4]{BBH21} (see also \cite[Theorem B]{BH19}), there exist a proper parabolic $k$-subgroup $\mathbf P < \mathbf Q < \mathbf G$ and a $\Gamma$-equivariant unital normal embedding $\iota : \rL^\infty(G/Q) \to M$ such that $\Phi \circ \iota : \rL^\infty(G/Q) \to \rL^\infty(G/P)$ is the canonical embedding with $Q = \mathbf Q(k)$. Then we have $\rL(\Lambda \curvearrowright G/Q) \subset M$. Since $\Lambda \curvearrowright G/Q$ is essentially free (see \cite[Lemma 6.2]{BBHP20}), a combination of \cite[Theorem IV.2.11]{Ma91} and  \cite[Theorem 3.6]{Su18} shows that there exists a unique parabolic $k$-subgroup $\mathbf Q < \mathbf R < \mathbf G$ such that $M = \rL(\Lambda \curvearrowright G/R)$ with $R = \mathbf R(k)$.

Secondly, assume that $d \geq 2$. Since $\Gamma < G$ need not be embedded with dense projections, we cannot apply Theorem \ref{thmA}. However, we may proceed as in the proof of \cite[Proposition 6.1]{BBHP20}. Upon permuting the indices, letting $H_1 = \overline{p_1(\Gamma)} < G_1$ and $H_2 = \overline{\widehat p_1(\Gamma)} < \prod_{j = 2}^d G_j$, we have that $\Gamma < H_1 \times H_2$ is a lattice embedded with dense projections and the von Neumann subalgebra $M_1 \subset M$ of all $H_1$-continuous elements in $M$ is nontrivial. Set $B_1 = G_1/P_1$ and $B_2 = \prod_{j = 2}^d G_j/P_j$. Since $G_1^+ < H_1$ and $\prod_{j = 2}^d G_j^+ < H_2$, as explained in the proof of \cite[Proposition 6.1]{BBHP20} and using \cite[Corollary 5.2]{BS04}, for every $i \in \{1, 2\}$, we may choose Borel probability measures $\mu_i \in \Prob(H_i)$ and $\nu_i \in \Prob(B_i)$ such that $(B_i, \nu_i)$ is the $(H_i, \mu_i)$-Furstenberg-Poisson boundary. Since $G_1^+ < H_1$ and $\prod_{j = 2}^d G_j^+ < H_2$, using \cite[Theorem I.1.5.6$(\rm i)$]{Ma91} and \cite[Proposition 4.3]{CM08}, we infer that for every $i \in \{1, 2\}$, $\mathscr{Q}\mathscr{Z}(H_i) = \mathscr Z(H_i)$. Moreover, \cite[Lemma 6.2]{BBHP20} implies that for every $i \in \{1, 2\}$, the pair $(H_i, \mu_i)$ satisfies the boundary freeness condition. We may now apply Theorem \ref{thmA} to obtain the conclusion.

For every $i \in \{1, \dots, d\}$, there are $2^{\rk_{k_i}(\mathbf G_i)}$ intermediate parabolic $k_i$-subgroups $\mathbf P_i < \mathbf Q_i < \mathbf G_i$. The NFT implies that there are $2^{\sum_{i = 1}^d \rk_{k_i}(\mathbf G_i)}$ intermediate von Neumann subalgebras $\rL(\Lambda) \subset M \subset \rL(\Lambda \curvearrowright G/P)$. Thus, the rank $\sum_{i = 1}^d \rk_{k_i}(\mathbf G_i)$ is an invariant of the inclusion $\rL(\Lambda) \subset \rL(\Lambda \curvearrowright G/P)$.
\end{proof}

\begin{proof}[Proof of Theorem \ref{thmC}]
For every $i \in \{1, \dots, d\}$, we have $\mathscr{Q}\mathscr{Z}(G_i) = \{e\}$ (see \cite[Lemma 3.1.1 and Proposition 3.1.2]{BM00a}). Moreover, the nonsingular action $\Gamma \curvearrowright \partial \mathsf T_i$ is essentially free (see \cite[Proposition 6.4]{BBHP20}). This condition is sufficient to apply the proof of Theorem \ref{thmA} to obtain the conclusion.
\end{proof}


\bibliographystyle{plain}

\end{document}